\documentclass[12pt,a4paper]{amsart}

\usepackage[T1]{fontenc}
\usepackage[left=25mm,top=35mm,right=25mm,bottom=26mm,headsep=10mm]{geometry}
\usepackage{microtype,textcomp,fbb}
\usepackage{
    amsmath,  amssymb,  amsthm,   amscd,
    gensymb,  graphicx, etoolbox, 
    booktabs, stackrel, mathtools    
}
\usepackage[usenames,dvipsnames]{xcolor}
\usepackage[colorlinks=true, linkcolor=blue, citecolor=blue, urlcolor=blue, breaklinks=true]{hyperref}
\usepackage[capitalise]{cleveref}
\newtheorem{theorem}{Theorem}[section]

\newtheorem{definition}[theorem]{Definition}
\newtheorem{example}[theorem]{Example}

\newtheorem{proposition}[theorem]{Proposition}
\newtheorem{remark}[theorem]{Remark}

\usepackage{tikz}
\usetikzlibrary{arrows}
\usetikzlibrary{shapes}
\usepackage{float}
\usepackage{tabularx}
\usepackage{kantlipsum}
\tikzset{edgee/.style = {> = latex'}}
\usepackage{array}
\newcolumntype{P}[1]{>{\centering\arraybackslash}p{#1}}
\newcolumntype{M}[1]{>{\centering\arraybackslash}m{#1}}

\newcommand{\treea}[4]{
\begin{scope}[shift={(#1)}]
\node [circle,draw=black](1) at (-0.25,-0.25) {#2};
\node [circle,draw=black,inner sep=2pt,fill=black](2) at (-1,1.75) {};
\node [circle,draw=black](3) at (1.5,2) {#3};
\node [circle,draw=black](4) at (3,4) {#4};
\node [circle,draw=black,inner sep=2pt,fill=black] (4') at (3-0.75,6) {};
\node [circle,draw=black,inner sep=2pt,fill=black] (4'') at (3+0.75,6) {};
\node [circle,draw=black,inner sep=2pt,fill=black](5) at (0.75,4) {};
\draw (1)--(2);
\draw (1)--(3);
\draw (3)--(4);
\draw (3)--(5);
\draw (4')--(4)--(4'');
\node at (1.25, -2) {Branches: $1$};
\end{scope}
}

\newcommand{\treeb}[5]{
\begin{scope}[shift={(#1)},xscale=-1]
\node [circle,draw=black](1) at (-0.25,-0.25) {#2};
\node [circle,draw=black,inner sep=2pt,fill=black](2) at (-1,1.75) {};
\node [circle,draw=black](3) at (1.5,2) {#3};
\node [circle,draw=black](4) at (3,4) {#4};
\node [circle,draw=black,inner sep=2pt,fill=black] (4') at (3-0.75,6) {};
\node [circle,draw=black,inner sep=2pt,fill=black] (4'') at (3+0.75,6) {};
\node [circle,draw=black,inner sep=2pt,fill=black](5) at (0.75,4) {};
\draw (1)--(2);
\draw (1)--(3);
\draw (3)--(4);
\draw (3)--(5);
\draw (4')--(4)--(4'');
\node at (1.25, -2) {Branches: #5};
\end{scope}
}

\newcommand{\treec}[4]{
\begin{scope}[shift={(#1)}]
\node [circle,draw=black](1) at (-0.25,-0.25) {#2};
        \node [circle,draw=black,inner sep=2pt,fill=black](2) at (-1.5,1.75) {};
        \node [circle,draw=black](3) at (1.5,2) {#3};
        \node [circle,draw=black,inner sep=2pt,fill=black](4) at (2.5,4) {};
        \node [circle,draw=black](5) at (0,4) {#4};
        \node [circle,draw=black,inner sep=2pt,fill=black](6) at (0.75,6) {};
        \node [circle,draw=black,inner sep=2pt,fill=black](7) at (-0.75,6) {};
        \draw (1)--(2);
        \draw (1)--(3);
        \draw (3)--(4);
        \draw (3)--(5);
        \draw (5)--(6);
        \draw (5)--(7);
        \node at (0.75, -2) {Branches: 1};
\end{scope}
}

\newcommand{\treed}[5]{
\begin{scope}[shift={(#1)},xscale=-1]
\node [circle,draw=black](1) at (-0.25,-0.25) {#2};
        \node [circle,draw=black,inner sep=2pt,fill=black](2) at (-1.5,1.75) {};
        \node [circle,draw=black](3) at (1.5,2) {#3};
        \node [circle,draw=black,inner sep=2pt,fill=black](4) at (2.5,4) {};
        \node [circle,draw=black](5) at (0,4) {#4};
        \node [circle,draw=black,inner sep=2pt,fill=black](6) at (0.75,6) {};
        \node [circle,draw=black,inner sep=2pt,fill=black](7) at (-0.75,6) {};
        \draw (1)--(2);
        \draw (1)--(3);
        \draw (3)--(4);
        \draw (3)--(5);
        \draw (5)--(6);
        \draw (5)--(7);
        \node at (0.5, -2) {Branches: #5};
\end{scope}
}

\newcommand{\treee}[5]{
\begin{scope}[shift={(#1)}]
        \node [circle,draw=black](3) at (1.5,2) {#3};
        \node [circle,draw=black](4) at (3,4) {#4};
        \node [circle,draw=black,inner sep=2pt,fill=black] (4') at (3-0.75,6) {};
        \node [circle,draw=black,inner sep=2pt,fill=black] (4'') at (3+0.75,6) {};
        \node [circle,draw=black](5) at (0,4) {#2};
        \node [circle,draw=black,inner sep=2pt,fill=black](6) at (0.75,6) {};
        \node [circle,draw=black,inner sep=2pt,fill=black](7) at (-0.75,6) {};
        \draw (3)--(4);
        \draw (3)--(5);
        \draw (5)--(6);
        \draw (5)--(7);
        \draw (4')--(4)--(4'');
        \node at (1.5, 0.25) {Branches: #5};
\end{scope}
}

\newcommand\R{\mathbb{R}}

\newcommand\N{\mathbb{N}}
\newcommand{\A}{\mathcal{A}}
\newcommand{\T}{\mathcal{T}}
\newcommand{\ipa}{\mathrm{L}}

\begin{document}

\title[A Branch Statistic for Trees]{A Branch Statistic for Trees: Interpreting Coefficients of the Characteristic Polynomial of Braid Deformations}
 \author{Priyavrat Deshpande}
 \address{Chennai Mathematical Institute}
 \email{pdeshpande@cmi.ac.in}
 \author{Krishna Menon}
 \address{Chennai Mathematical Institute}
 \email{krishnamenon@cmi.ac.in}

\begin{abstract}
A hyperplane arrangement in $\mathbb{R}^n$ is a finite collection of affine hyperplanes.
The regions are the connected components of the complement of these hyperplanes.
By a theorem of Zaslavsky, the number of regions of a hyperplane arrangement is the sum of the absolute values of the coefficients of its characteristic polynomial.
Arrangements that contain hyperplanes parallel to subspaces whose defining equations are $x_i - x_j = 0$ form an important class called the deformations of the braid arrangement.
In a recent work, Bernardi showed that regions of certain deformations are in one-to-one correspondence with certain labeled trees.
In this article, we define a statistic on these trees such that the distribution is given by the coefficients of the characteristic polynomial.
In particular, our statistic applies to well-studied families like extended Catalan, Shi, Linial and semiorder.
\end{abstract}

\maketitle

\section{Introduction}\label{intro}

A \emph{hyperplane arrangement} $\A$ is a finite collection of affine hyperplanes (i.e., codimension $1$ subspaces and their translates) in $\R^n$.
A \emph{region} of $\A$ is a connected component of $\R^n\setminus \bigcup \A$.
The number of regions of $\A$ is denoted by $r(\mathcal{A})$.
The poset of non-empty intersections of hyperplanes in an arrangement $\mathcal{A}$ ordered by reverse inclusion is called its \emph{intersection poset} denoted by $\ipa(\A)$.
The ambient space of the arrangement (i.e., $\mathbb{R}^n$) is an element of the intersection poset; considered as the intersection of none of the hyperplanes.
The \emph{characteristic polynomial} of $\A$ is defined as 
\[\chi_\A (t) := \sum_{x\in\ipa(\A)} \mu(\hat{0},x)\, t^{\dim(x)}\]
where $\mu$ is the M\"obius function of the intersection poset and $\hat{0}$ corresponds to $\R^n$.
Using the fact that every interval of the intersection poset of an arrangement is a geometric lattice, we have
\begin{equation}\label{charform}
    \chi_\A(t) = \sum_{i=0}^n (-1)^{n-i} c_i t^i
\end{equation}
where $c_i$ is a non-negative integer for all $0 \leq i \leq n$ \cite[Corollary 3.4]{stanarr07}.
The characteristic polynomial is a fundamental combinatorial and topological invariant of the arrangement and plays a significant role throughout the theory of hyperplane arrangements.

In this article, our focus is on the enumerative aspects of (rational) arrangements in $\R^n$. 
In that direction we have the following seminal result by Zaslavsky.

\begin{theorem}[\cite{zas75}]\label{zas}
Let $\A$ be an arrangement in $\R^n$. Then the number of regions of $\A$ is given by 
\begin{align*}
   r(\A) &= (-1)^n \chi_{\A}(-1) \\
         &= \sum_{i=0}^n c_i. 
\end{align*}
\end{theorem}

When the regions of an arrangement are in bijection with a certain combinatorially defined set, one could ask if there is a corresponding `statistic' on the set whose distribution is given by the $c_i$'s.
For example, the regions of the braid arrangement in $\R^n$ (whose hyperplanes are given by the equations $x_i-x_j = 0$ for $1\leq i<j\leq n$) correspond to the $n!$ permutations of $[n]$.
The characteristic polynomial of this arrangement is 
$t(t-1)\cdots(t-n+1)$ \cite[Corollary 2.2]{stanarr07}. 
Hence, $c_i$'s are the unsigned Stirling numbers of the first kind. 
Consequently, the distribution of the statistic `number of cycles' on the set of permutations is given by the coefficients of the characteristic polynomial.

In this paper, we consider arrangements where each hyperplane is of the form $x_i-x_j=s$ for some $s \in \mathbb{Z}$.
Such arrangements are called \textit{deformations of the braid arrangement}.
Recently, Bernardi \cite{ber} obtained a method to count the regions of any deformation of the braid arrangement using certain objects called \textit{boxed trees}.
For certain special deformations, which he calls \textit{transitive}, he also obtained an explicit bijection between the regions of the arrangement and a certain set of trees.
Our main aim is to obtain a statistic on such trees whose distribution is given by the coefficients of the characteristic polynomial of the corresponding arrangement.

For any finite set of integers $S$, we associate a deformation of the braid arrangement $\A_S(n)$ in $\R^n$ with hyperplanes
\begin{equation*}
    \{x_i-x_j=k \mid  k \in S,\ 1 \leq i<j \leq n\}.
\end{equation*}
Important examples of such arrangements are the Catalan, Shi, Linial and semiorder arrangements.
These correspond to $S=\{-1,0,1\}$, $\{0,1\}$, $\{1\}$, and $\{-1,1\}$ respectively.
For any $m \geq 1$, the extended Catalan arrangement, or $m$-Catalan arrangement, in $\R^n$ is $\A_S(n)$ where $S=\{-m,\ldots,m\}$.
Similarly, the extended Shi, Linial, and semiorder arrangements correspond to $S=\{-m+1,\ldots,m\}$, $\{-m+1,\ldots,m\} \setminus \{0\}$, and $\{-m,\ldots,m\} \setminus \{0\}$ respectively.

If the set $S$ satisfies certain conditions (see \Cref{transdef}), then the arrangements $\A_S(n)$ are called transitive.
The extended Catalan, Shi, Linial, and semiorder arrangements are all transitive.
We note here that Bernardi \cite{ber} considers a larger class of arrangements to be transitive, but we only focus on arrangements of the form described above.

From \cite[Theorem 3.8]{ber}, we know that if $S$ is transitive, then the regions of $\A_S(n)$ are in bijection with a certain set of trees $\T_S(n)$ (see \Cref{Streesdef}).
For example, when $S=\{0,1\}$ which corresponds to the Shi arrangement, $\T_{\{0,1\}}(n)$ is the set of labeled binary trees with $n$ nodes where any right node has a label smaller than its parent.

\begin{example}
A tree in $\T_{\{0,1\}}(4)$ is shown in \Cref{binarytree}.
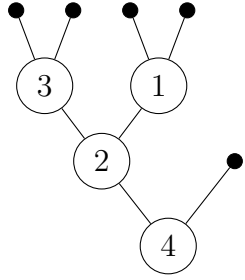
\begin{figure}[!htbp]
        \centering
        \begin{tikzpicture}[xscale=-1,scale=0.5]
       \node [circle,draw=black](1) at (-0.25,-0.25) {$4$};
        \node [circle,draw=black,inner sep=2pt,fill=black](2) at (-2,2) {};
        \node [circle,draw=black](3) at (1.5,2) {$2$};
        \node [circle,draw=black](4) at (3,4) {$3$};
        \node [circle,draw=black,inner sep=2pt,fill=black] (4') at (3-0.75,6) {};
        \node [circle,draw=black,inner sep=2pt,fill=black] (4'') at (3+0.75,6) {};
        \node [circle,draw=black](5) at (0,4) {$1$};
        \node [circle,draw=black,inner sep=2pt,fill=black](6) at (0.75,6) {};
        \node [circle,draw=black,inner sep=2pt,fill=black](7) at (-0.75,6) {};
        \draw (1)--(2);
        \draw (1)--(3);
        \draw (3)--(4);
        \draw (3)--(5);
        \draw (5)--(6);
        \draw (5)--(7);
        \draw (4')--(4)--(4'');
    \end{tikzpicture}
        \caption{A tree in $\T_{\{0,1\}}(4)$}
        \label{binarytree}
    \end{figure}
\end{example}

For such a tree, we define the \textit{trunk} to be the path from the root to the leftmost leaf.
Using the nodes on the trunk, we obtain a sequence of numbers.
A node in this sequence that is greater than all the nodes after it is called a \textit{branch node}.
For the tree in \Cref{binarytree}, the sequence on the trunk is $4,2,3$ and the branch nodes are $4$ and $3$.
These definitions can be generalized to trees in $\T_S(n)$ for other sets $S$.

The main theorem of this article is:
\begin{theorem}
    For a transitive set $S$, the absolute value of the coefficient of $t^j$ in $\chi_{\A_S(n)}(t)$ is the number of trees in $\T_S(n)$ with $j$ branch nodes.
\end{theorem}

The article begins with a short account of Bernardi's work \cite{ber} in \Cref{prelim}.
In \cref{bstat} the branch statistic is introduced and the main theorem is proved.
In \cref{mcat} we derive some properties of the coefficients of the characteristic polynomial and in particular study the extended Catalan arrangements.

An extended abstract of this article \cite{extabs} will appear in a proceedings volume of S\'{e}minaire Lotharingien Combinatoire for FPSAC 2022.

\section{Preliminaries}\label{prelim}

A \textit{tree} is a graph with no cycles.
A \textit{rooted tree} is a tree with a distinguished vertex called the root.
We will draw rooted trees with their root at the bottom.
Children of a vertex $v$ in a rooted tree are those vertices $w$ that are adjacent to $v$ and such that the unique path from the root to $w$ passes through $v$.
Similarly, we can define the parent of a vertex $v$ to be the vertex $w$ for which $v$ is the child of $w$.
Any non-root vertex has a unique parent.
All the vertices that have at least one child are called \textit{nodes} and those that do not are called \textit{leaves}.

A \textit{rooted plane tree} is a rooted tree with a specified ordering for the children of each node.
When drawing a rooted plane tree, the children of any node will be ordered from left to right.
The \textit{left siblings} of a vertex $v$ are the vertices that are also children of the parent of $v$ but are to the left of $v$.
We denote the number of left siblings of $v$ as $\operatorname{lsib}(v)$.

\begin{definition}
    An \textit{$(m+1)$-ary tree} is a rooted plane tree where each node has exactly $(m+1)$ children.
    We will denote by $\mathcal{T}^{(m)}(n)$ the set of all $(m+1)$-ary trees with $n$ nodes labeled with distinct elements from $[n]$.
\end{definition}

For trees in $\mathcal{T}^{(m)}(n)$, we will denote the node having label $i \in [n]$ by just $i$.

\begin{definition}
    If a node $i$ in a tree $T \in \mathcal{T}^{(m)}(n)$ has at least one child that is a node, the \emph{cadet} of $i$ is the rightmost such child, which we denote by $\operatorname{cadet}(i)$.
\end{definition}

\begin{example}
\Cref{binarytree} shows an element of $\mathcal{T}^{(1)}(4)$ where
\begin{itemize}
    \item $4$ is the root,
    \item  $\operatorname{lsib}(2)=0$, $\operatorname{lsib}(3)=0$, $\operatorname{lsib}(1)=1$,
    \item $\operatorname{cadet}(4)=2$, and $\operatorname{cadet}(2)=1$.
\end{itemize}
\end{example}

\begin{definition}\label{Streesdef}
    For any finite set of integers $S$ with $m$ = $\operatorname{max}\{ |s| \mid  s\in S\}$, define $\mathcal{T}_S(n)$ to be the set of trees in $\mathcal{T}^{(m)}(n)$, such that if $\operatorname{cadet}(i)=j$:
    \begin{itemize}
        \item $\operatorname{lsib}(j) \notin S \cup \{0\}$ $\Rightarrow$ $i<j$.
        \item $-\operatorname{lsib}(j) \notin S$ $\Rightarrow$ $i>j$.
    \end{itemize}
\end{definition}

Recall that for any finite set of integers $S$, we defined the arrangement $\mathcal{A}_S(n)$ as the deformation of the braid arrangement in $\mathbb{R}^n$ with hyperplanes
\begin{equation*}
    \{x_i-x_j=k \mid  k \in S,\ 1 \leq i<j \leq n\}.
\end{equation*}
Though Bernardi \cite{ber} derived results for more general deformations, we will only be focused on these.

\begin{definition}\label{transdef}
    A finite set of integers $S$ is said to be \textit{transitive} if for any $s,t \notin S$,
    \begin{itemize}
        \item $st>0$ $\Rightarrow$ $s+t \notin S$.
        \item $s>0$ and $t\leq 0$ $\Rightarrow$ $s-t \notin S$ and $t-s \notin S$.
    \end{itemize}
\end{definition}

\begin{example}
As mentioned in \Cref{intro}, for any $m \geq 1$, the sets $\{-m,\ldots,m\}$, $\{-m+1,\ldots,m\}$, $\{-m,\ldots,m\} \setminus \{0\}$, and $\{-m+1,\ldots,m\} \setminus \{0\}$ are all transitive.
\end{example}

We can now state the result for arrangements $\mathcal{A}_S(n)$ where $S$ is transitive.

\begin{theorem}{\cite[Theorem 3.8]{ber}}
    For any transitive set of integers $S$, the regions of the arrangement $\mathcal{A}_S(n)$ are in bijection with the trees in $\mathcal{T}_S(n)$.
\end{theorem}

Before looking at the characteristic polynomials of such arrangements, we recall a few results from \cite{ec2}.
Suppose that $c: \N \rightarrow \N$ is a function and for each $n,j \in \N$, we define
\begin{equation*}
    c_j(n)=\sum_{\{B_1,\ldots,B_j\} \in \Pi_n} c(|B_1|)\cdots c(|B_j|)
\end{equation*}
where $\Pi_n$ is the set of partitions of $[n]$.
Define for each $n \in \N$,
\begin{equation*}
    h(n) = \sum_{j=0}^{n}c_j(n).
\end{equation*}
From \cite[Example 5.2.2]{ec2}, we know that in such a situation,
\begin{equation*}
    \sum_{n,j\geq 0}c_j(n)t^j\frac{x^n}{n!} = \left(\sum_{n\geq 0}h(n)\frac{x^n}{n!}\right)^t.
\end{equation*}
Informally, we consider $h(n)$ to be the number of ``structures'' that can be placed on an $n$-set where each structure can be uniquely broken up into a disjoint union of ``connected sub-structures''.
Here $c(n)$ denotes the number of connected structures on an $n$-set and $c_j(n)$ denotes the number of structures on an $n$-set with exactly $j$ connected sub-structures.

We now consider the characteristic polynomials of arrangements of the form $\A_S(n)$.
For a fixed set $S$, the sequence of arrangements $(\A_S(1),\A_S(2),\ldots)$ forms what is called an \textit{exponential sequence of arrangements} (ESA).

\begin{definition}{\cite[Definition 5.14]{stanarr07}}
A sequence of arrangements $(\A_1,\A_2,\ldots)$ is called an ESA if
\begin{itemize}
    \item $\A_n$ is a braid deformation in $\R^n$.
    \item For any $k$-subset $I$ of $[n]$, the arrangement
    \begin{equation*}
        \A_n^I = \{H \in \A_n \mid H \text{ is of the form $x_i-x_j=s$ for some $i,j \in I$}\}
    \end{equation*}
    satisfies $\ipa(\A_n^I) \cong \ipa(\A_k)$ (isomorphic as posets).
\end{itemize}
\end{definition}

The result on ESAs that we will need is the following.

\begin{theorem}{\cite[Theorem 5.17]{stanarr07}}
    If $(\A_1,\A_2,\ldots)$ is an ESA, then
    \begin{equation*}
        \sum_{n\geq 0}\chi_{\A_n}(t)\frac{x^n}{n!} = \left(\sum_{n\geq 0}(-1)^nr(\A_n)\frac{x^n}{n!}\right)^{-t}.
    \end{equation*}
\end{theorem}

\begin{remark}
We note that this is also a special case of \cite[Theorem 5.2]{ber}.
\end{remark}

Using this result, the form of a characteristic polynomial given in \eqref{charform}, and the above discussion on connected structures, we note that interpreting the coefficients of the polynomial $\chi_{\A_S(n)}(t)$ is equivalent to defining a notion of ``connected structures'' for trees in $\T_S(n)$.
We do this in the next section.

\section{A branch statistic}\label{bstat}

A \textit{label set} is a finite set of positive integers.
For any label set $V$, we define $\T^{(m)}(V)$ to be the set of $(m+1)$-ary trees with $|V|$ nodes labeled distinctly using $V$.
Note that $\T^{(m)}([n])=\T^{(m)}(n)$.

We now describe the method we use to break up a tree in $\T^{(m)}(V)$ into ``connected sub-structures'', which we call \textit{branches}.

\begin{definition}
The \textit{trunk} of a tree in $\T^{(m)}(V)$ is the path from the root to the leftmost leaf.
The nodes on the trunk of the tree break up the tree into sub-trees, which we call \textit{twigs} (see \Cref{3arytree}).
\end{definition}

Let the nodes on the trunk of a tree be $v_1,v_2,\ldots,v_k$, where $v_1$ is the root and $v_{i+1}$ is the leftmost child of $v_i$ for any $i \in [k-1]$.
If $v_i = \operatorname{max}\{v_1,\ldots,v_k\}$, then the first branch of the tree consists of the twigs corresponding to the nodes $v_1,\ldots,v_i$.
If $v_j = \operatorname{max} \{v_{i+1},\ldots,v_k\}$, then the second branch of the tree consists of the twigs corresponding to the nodes $v_{i+1},\ldots,v_j$.
Continuing this way, we break up the tree into branches.

Note that the number of branches of the tree is just the number of right-to-left maxima of the sequence $v_1,v_2,\ldots,v_k$ of nodes on the trunk, \textit{i.e.}, the number of $v_i$ such that $v_i > v_j$ for all $j>i$.
We will call such $v_i$ the \textit{branch nodes} of the trunk.

\begin{example}\label{trunkexamp}
The tree in \Cref{3arytree} has $3$ twigs and $2$ branches.
The first branch consists of just the first twig since $6$ is the largest node in the trunk.
The second branch consists of the second and third twigs since $5$ is larger than $4$.
Here $6$ and $5$ are the branch nodes.
\end{example}

\begin{figure}[!htbp]
    \centering
    \begin{tikzpicture}[rotate=180]
    \draw [dashed,blue,thin] (1.25,0.5) rectangle (5,-5.25);
    \draw [dashed,red,thin] (1.25-0.25,0.5+0.25) rectangle (5+0.25,-5.25-0.25);
    \draw [dashed,blue,thin] (6,-1.5) rectangle (8,-5.25);
    \draw [dashed,blue,thin] (8.5,-3.5) rectangle (11,-6.25);
    \draw [dashed,red,thin] (6-0.25,-1.5+0.25) rectangle (11+0.25,-6.25-0.25);
    \node (1) [draw=black,circle] at (4,0) {$6$};
    \node (2) [draw=black,circle] at (1+1,-2) {$2$};
    \node (3) [draw=black,circle] at (4,-2) {$3$};
    \node (4) [draw=black,circle] at (7,-2) {$4$};
    \node (5) [circle,fill=black,inner sep=2pt] at (0.5+1,-3) {};
    \node (6) [circle,fill=black,inner sep=2pt] at (3.5,-3) {};
    \node (7) [circle,fill=black,inner sep=2pt] at (4,-3) {};
    \node (8) [circle,fill=black,inner sep=2pt] at (4.5,-3) {};
    \node (9) [circle,fill=black,inner sep=2pt] at (6.5,-3) {};
    \node (10) [draw=black,circle] at (10,-4) {$5$};
    \node (11) [draw=black,circle] at (1+1,-4) {$7$};
    \node (12) [draw=black,circle] at (7,-4) {$1$};
    \node (13) [draw=black,circle] at (10-1+0.25,-5) {$8$};
    \node (14) [circle,fill=black,inner sep=2pt] at (11-1,-5) {};
    \node (15) [circle,fill=black,inner sep=2pt] at (0.5+1,-5) {};
    \node (16) [circle,fill=black,inner sep=2pt] at (1+1,-5) {};
    \node (17) [circle,fill=black,inner sep=2pt] at (1.5+1,-5) {};
    \node (18) [circle,fill=black,inner sep=2pt] at (6.5,-5) {};
    \node (19) [circle,fill=black,inner sep=2pt] at (7,-5) {};
    \node (20) [circle,fill=black,inner sep=2pt] at (7.5,-5) {};
    \node (21) [circle,fill=black,inner sep=2pt] at (9-1+0.5+0.25,-6) {};
    \node (22) [circle,fill=black,inner sep=2pt] at (10-1+0.25,-6) {};
    \node (23) [circle,fill=black,inner sep=2pt] at (11-1-0.5+0.25,-6) {};
    \node (24) [circle,fill=black,inner sep=2pt] at (10.5+0.25,-5) {};
    \node (25) [circle,fill=black,inner sep=2pt] at (1.5+1,-3) {};
    \draw (5)--(2)--(1)--(4)--(10)--(24);
    \draw (15)--(11)--(2);
    \draw (16)--(11);
    \draw (17)--(11);
    \draw (6)--(3)--(7);
    \draw (8)--(3);
    \draw (9)--(4)--(12)--(18);
    \draw (19)--(12)--(20);
    \draw (21)--(13)--(10)--(14);
    \draw (22)--(13)--(23);
    \draw (3)--(1);
    \draw (2)--(25);
    \end{tikzpicture}
    \caption{A labeled $3$-ary tree with twigs and branches specified.}
    \label{3arytree}
\end{figure}
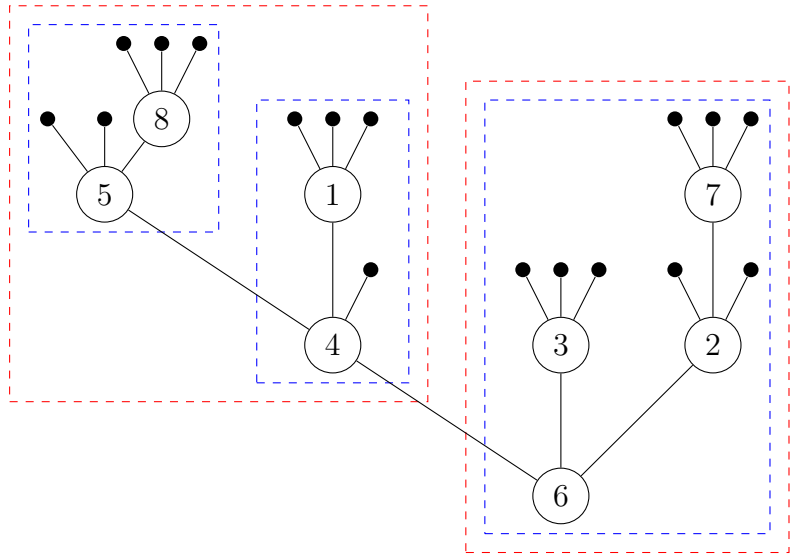

We use the notation $\T_j^{(m)}(V)$ to denote the trees in $\T^{(m)}(V)$ having $j$ branches.
To prove that this is indeed a break-up of trees into connected sub-structures, we have to prove that
\begin{equation}\label{expstructeq}
    |\T_j^{(m)}(V)| = \sum_{\{B_1,\ldots,B_j\} \in \Pi_V} |\T_1^{(m)}(B_1)|\cdots|\T_1^{(m)}(B_j)|.
\end{equation}
Hence, ``connected'' trees are those with exactly one branch, \textit{i.e.}, trees where the last node of the trunk is the one with the largest label.

The connected components associated to a given tree are the branches of the tree.

\begin{example}
The connected components associated to the tree in \Cref{3arytree} are given in \Cref{conncomp}.
\end{example}

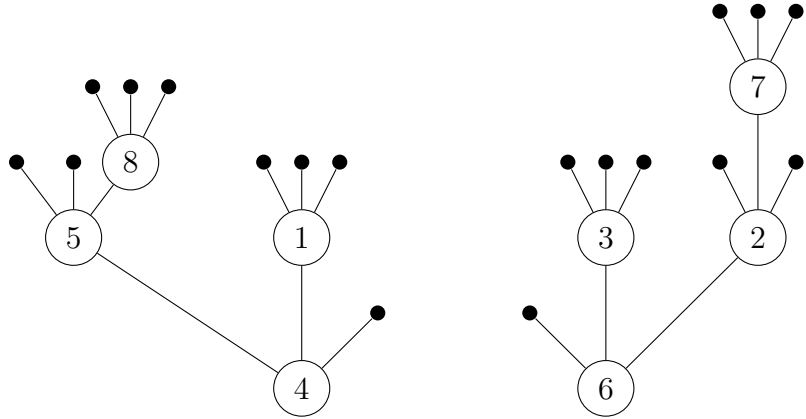
\begin{figure}[!htbp]
    \centering
    \begin{tikzpicture}[rotate=180]
    \node (1) [draw=black,circle] at (4,0) {$6$};
    \node (a) [circle,fill=black,inner sep=2pt] at (5,-1) {};
    \draw (1)--(a);
    \node (2) [draw=black,circle] at (1+1,-2) {$2$};
    \node (3) [draw=black,circle] at (4,-2) {$3$};
    \node (4) [draw=black,circle] at (7+1,-2+2) {$4$};
    \node (5) [circle,fill=black,inner sep=2pt] at (0.5+1,-3) {};
    \node (6) [circle,fill=black,inner sep=2pt] at (3.5,-3) {};
    \node (7) [circle,fill=black,inner sep=2pt] at (4,-3) {};
    \node (8) [circle,fill=black,inner sep=2pt] at (4.5,-3) {};
    \node (9) [circle,fill=black,inner sep=2pt] at (7,-3+2) {};
    \node (10) [draw=black,circle] at (10+1,-4+2) {$5$};
    \node (11) [draw=black,circle] at (1+1,-4) {$7$};
    \node (12) [draw=black,circle] at (7+1,-4+2) {$1$};
    \node (13) [draw=black,circle] at (10-1+0.25+1,-5+2) {$8$};
    \node (14) [circle,fill=black,inner sep=2pt] at (11-1+1,-5+2) {};
    \node (15) [circle,fill=black,inner sep=2pt] at (0.5+1,-5) {};
    \node (16) [circle,fill=black,inner sep=2pt] at (1+1,-5) {};
    \node (17) [circle,fill=black,inner sep=2pt] at (1.5+1,-5) {};
    \node (18) [circle,fill=black,inner sep=2pt] at (6.5+1,-5+2) {};
    \node (19) [circle,fill=black,inner sep=2pt] at (7+1,-5+2) {};
    \node (20) [circle,fill=black,inner sep=2pt] at (7.5+1,-5+2) {};
    \node (21) [circle,fill=black,inner sep=2pt] at (9-1+0.5+0.25+1,-6+2) {};
    \node (22) [circle,fill=black,inner sep=2pt] at (10-1+0.25+1,-6+2) {};
    \node (23) [circle,fill=black,inner sep=2pt] at (11-1-0.5+0.25+1,-6+2) {};
    \node (24) [circle,fill=black,inner sep=2pt] at (10.5+0.25+1,-5+2) {};
    \node (25) [circle,fill=black,inner sep=2pt] at (1.5+1,-3) {};
    \draw (5)--(2)--(1);
    \draw (4)--(10)--(24);
    \draw (15)--(11)--(2);
    \draw (16)--(11);
    \draw (17)--(11);
    \draw (6)--(3)--(7);
    \draw (8)--(3);
    \draw (9)--(4)--(12)--(18);
    \draw (19)--(12)--(20);
    \draw (21)--(13)--(10)--(14);
    \draw (22)--(13)--(23);
    \draw (3)--(1);
    \draw (2)--(25);
    \end{tikzpicture}
    \caption{Connected components of the tree in \Cref{3arytree}.}
    \label{conncomp}
\end{figure}

A collection of connected trees (with disjoint label sets) can be put together in exactly one way to form a tree for which they form the branches.
This is done as follows:
Find the largest label among those on the trunks of the connected trees.
The connected tree $T_1$ with this label is made the first branch of the tree we are building.
Again, find the largest label among those on the trunks of the remaining connected trees.
The connected tree $T_2$ with this label is made the second branch of the tree we are building by gluing it to $T_1$.
This is done by deleting the leftmost leaf of $T_1$ and fixing the root of $T_2$ in its position.
This process is repeated to until all the connected trees are glued together.

\begin{example}
The tree associated to the collection of connected trees in \Cref{collconn} is given in \Cref{built3arytree}.
\end{example}

\begin{figure}[!htbp]
    \centering
    \begin{tikzpicture}[rotate=180]
    \node (1) [draw=black,circle] at (4,0) {$6$};
    \node (a) [circle,fill=black,inner sep=2pt] at (5,-1) {};
    \draw (1)--(a);
    \node (2) [draw=black,circle] at (1+1,-2) {$2$};
    \node (3) [draw=black,circle] at (4,-2) {$3$};
    \node (4) [draw=black,circle] at (7+1,-2+2) {$4$};
    \node (5) [circle,fill=black,inner sep=2pt] at (0.5+1,-3) {};
    \node (b) [circle,fill=black,inner sep=2pt] at (2,-3) {};
    \node (6) [circle,fill=black,inner sep=2pt] at (3.5,-3) {};
    \node (7) [circle,fill=black,inner sep=2pt] at (4,-3) {};
    \node (8) [circle,fill=black,inner sep=2pt] at (4.5,-3) {};
    \node (9) [circle,fill=black,inner sep=2pt] at (7,-3+2) {};
    \node (10) [draw=black,circle] at (10+1-0.75,-4+2) {$5$};
    \node (11) [draw=black,circle] at (1-1,-4+4) {$7$};
    \node (12) [draw=black,circle] at (7+1,-4+2) {$1$};
    \node (13) [draw=black,circle] at (10-1+0.25+1-0.75,-5+2) {$8$};
    \node (14) [circle,fill=black,inner sep=2pt] at (11-1+1-0.75,-5+2) {};
    \node (15) [circle,fill=black,inner sep=2pt] at (0.5-1,-5+4) {};
    \node (16) [circle,fill=black,inner sep=2pt] at (1-1,-5+4) {};
    \node (17) [circle,fill=black,inner sep=2pt] at (1.5-1,-5+4) {};
    \node (18) [circle,fill=black,inner sep=2pt] at (6.5+1,-5+2) {};
    \node (19) [circle,fill=black,inner sep=2pt] at (7+1,-5+2) {};
    \node (20) [circle,fill=black,inner sep=2pt] at (7.5+1,-5+2) {};
    \node (21) [circle,fill=black,inner sep=2pt] at (9-1+0.5+0.25+1-0.75,-6+2) {};
    \node (22) [circle,fill=black,inner sep=2pt] at (10-1+0.25+1-0.75,-6+2) {};
    \node (23) [circle,fill=black,inner sep=2pt] at (11-1-0.5+0.25+1-0.75,-6+2) {};
    \node (24) [circle,fill=black,inner sep=2pt] at (10.5+0.25+1-0.75,-5+2) {};
    \node (25) [circle,fill=black,inner sep=2pt] at (1.5+1,-3) {};
    \draw (5)--(2)--(1);
    \draw (4)--(10)--(24);
    \draw (15)--(11);
    \draw (b)--(2);
    \draw (16)--(11);
    \draw (17)--(11);
    \draw (6)--(3)--(7);
    \draw (8)--(3);
    \draw (9)--(4)--(12)--(18);
    \draw (19)--(12)--(20);
    \draw (21)--(13)--(10)--(14);
    \draw (22)--(13)--(23);
    \draw (3)--(1);
    \draw (2)--(25);
    \end{tikzpicture}
    \caption{A collection of connected trees.}
    \label{collconn}
\end{figure}
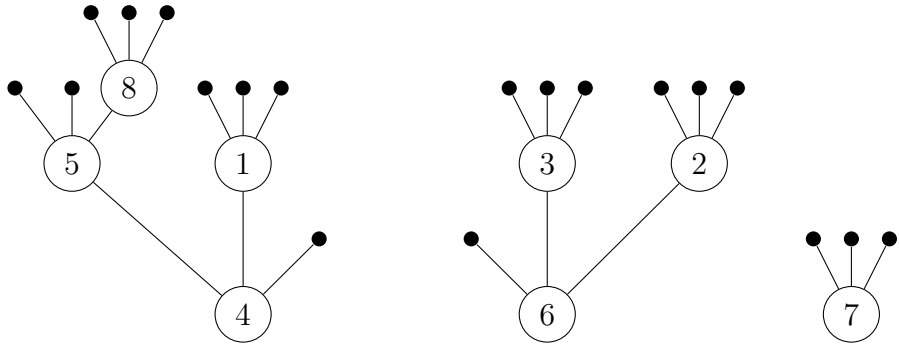

\begin{figure}[!htbp]
    \centering
    \begin{tikzpicture}[scale = 0.85, rotate=180]
    \node (1) [draw=black,circle] at (4,0) {$6$};
    \node (11) [draw=black,circle] at (1-1+3-0.5,0.8) {$7$};
    \draw (1)--(11);
    \node (15) [circle,fill=black,inner sep=2pt] at (0.5-1+3-0.5,0.75-5+4) {};
    \node (16) [circle,fill=black,inner sep=2pt] at (1-1+3-0.5,0.75-5+4) {};
    \draw (15)--(11)--(16);
    \node (2) [draw=black,circle] at (1+1,-2) {$2$};
    \node (b) [circle,fill=black,inner sep=2pt] at (2,-3) {};
    \node (3) [draw=black,circle] at (4,-2) {$3$};
    \node (4) [draw=black,circle] at (7-0.5,-0.5-2+1) {$4$};
    \node (5) [circle,fill=black,inner sep=2pt] at (0.5+1,-3) {};
    \node (6) [circle,fill=black,inner sep=2pt] at (3.5,-3) {};
    \node (7) [circle,fill=black,inner sep=2pt] at (4,-3) {};
    \node (8) [circle,fill=black,inner sep=2pt] at (4.5,-3) {};
    \node (9) [circle,fill=black,inner sep=2pt] at (6.5-0.5,-0.5-3+1) {};
    \node (10) [draw=black,circle] at (10-0.5-0.75,-0.25-4+1) {$5$};
    \node (12) [draw=black,circle] at (7-0.5,-0.5-4+1) {$1$};
    \node (13) [draw=black,circle] at (10-1+0.25-0.5-0.75,-0.5-5+1) {$8$};
    \node (14) [circle,fill=black,inner sep=2pt] at (11-1-0.5-0.75,-0.5-5+1) {};
    \node (18) [circle,fill=black,inner sep=2pt] at (6.5-0.5,-0.5-5+1) {};
    \node (19) [circle,fill=black,inner sep=2pt] at (7-0.5,-0.5-5+1) {};
    \node (20) [circle,fill=black,inner sep=2pt] at (7.5-0.5,-0.5-5+1) {};
    \node (21) [circle,fill=black,inner sep=2pt] at (9-1+0.5+0.25-0.5-0.75,-0.5-6+1) {};
    \node (22) [circle,fill=black,inner sep=2pt] at (10-1+0.25-0.5-0.75,-0.5-6+1) {};
    \node (23) [circle,fill=black,inner sep=2pt] at (11-1-0.5+0.25-0.5-0.75,-0.5-6+1) {};
    \node (24) [circle,fill=black,inner sep=2pt] at (10.5+0.25-0.5-0.75,-0.5-5+1) {};
    \node (25) [circle,fill=black,inner sep=2pt] at (1.5+1,-3) {};
    \draw (5)--(2)--(1)--(4)--(10)--(24);
    \draw (2)--(b);
    \draw (6)--(3)--(7);
    \draw (8)--(3);
    \draw (9)--(4)--(12)--(18);
    \draw (19)--(12)--(20);
    \draw (21)--(13)--(10)--(14);
    \draw (22)--(13)--(23);
    \draw (3)--(1);
    \draw (2)--(25);
    \end{tikzpicture}
    \caption{The tree associated to the collection of connected trees in \Cref{collconn}.}
    \label{built3arytree}
\end{figure}
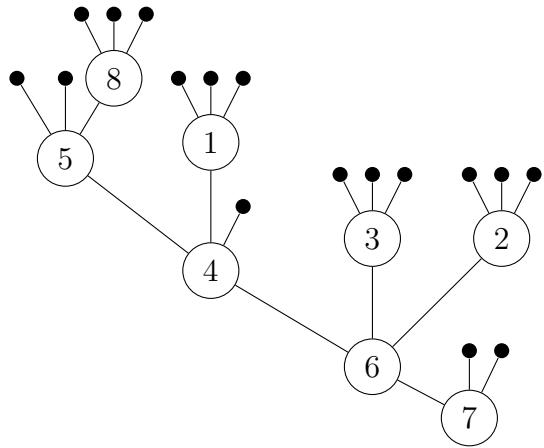

Recall that for a finite set of integers $S$ with $m=\operatorname{max}\{|s| : s \in S\}$, the set $\mathcal{T}_S(V)$, for some label set $V$, is the set of trees in $\mathcal{T}^{(m)}(V)$ such that if $\operatorname{cadet}(u)=v$, then
\begin{itemize}
    \item if $\operatorname{lsib}(v) \notin S \cup \{0\}$, we must have $u<v$, and
    \item if $-\operatorname{lsib}(v) \notin S$, we must have $u>v$.
\end{itemize}
We call this set of conditions ``Condition $S$''.

We set $\T_S:=\bigcup\limits_{V}\T_S(V)$ where the union is over all label sets $V$.
We now show that
\begin{enumerate}
    \item the connected components of any tree in $\T_S$ are also in $\T_S$, and
    \item trees that are built using connected trees in $\T_S$ are also in $\T_S$.
\end{enumerate}

We first note that statement 1 follows since the condition for a tree to be in $\T_S$ is a local condition.
This is also because, if $T'$ is a connected component of the tree $T$, the cadet of any node in $T'$ (if it exists) is the same as its cadet when considered as a node of $T$.

To prove statement 2, we only have to check that Condition $S$ is satisfied for the branch nodes of a tree built using connected trees in $\T_S$.
If a branch node does not have a cadet, Condition $S$ is trivially satisfied.
If a branch node $u$ has a cadet $v$, we consider two cases:
\begin{itemize}
    \item If the cadet is not the first child, then Condition $S$ is satisfied since it is satisfied by the connected components of the tree.
    \item If the cadet is the first child, then we must have $u>v$ since $u$ is a branch node.
    This makes sure that Condition $S$ is satisfied since we have $\operatorname{lsib}(v)=0$ and hence $\operatorname{lsib}(v) \in S \cup \{0\}$.
\end{itemize}

From the preceding, we get an equation analogous to \eqref{expstructeq} for the trees $\T_S$. 
Hence, from the discussion in \Cref{prelim}, we get the following result.

\begin{theorem}\label{statthm}
    For a transitive set of integers $S$, the absolute value of the coefficient of $t^j$ in $\chi_{\A_S(n)}(t)$ is the number of trees in $\T_S(n)$ with $j$ branches.
\end{theorem}

\begin{example}
When $S=\{0\}$, we obtain the braid arrangement.
Here, $\T_{\{0\}}(n)$ corresponds to permutations of $[n]$ and \Cref{statthm} states that the absolute value of the coefficient of $t^j$ in $\chi_{\A_{\{0\}}(n)}(t)$ is the number of permutations of $[n]$ with $j$ right-to-left maxima.
By \cite[Corollary 1.3.11]{ec1}, this agrees with the observation in \Cref{intro} that the coefficients are the Stirling numbers of the first kind.
\end{example}

\begin{example}
The Shi arrangement $\mathcal{S}_n$ in $\R^n$ is the deformation $\A_{\{0,1\}}(n)$.
The trees in $\T_{\{0,1\}}(n)$, called Shi trees, are those labeled binary trees where any right node has a label less than that of its parent.
The Shi trees for $n=3$ are given in \Cref{shi,lin}.
Counting the branches in these trees, we get $\chi_{\mathcal{S}_3}(t)=t^3+6t^2+9t$, which agrees with the known formula for the characteristic polynomial (for example, see \cite[Theorem 3.3]{athanas96}).
\end{example}

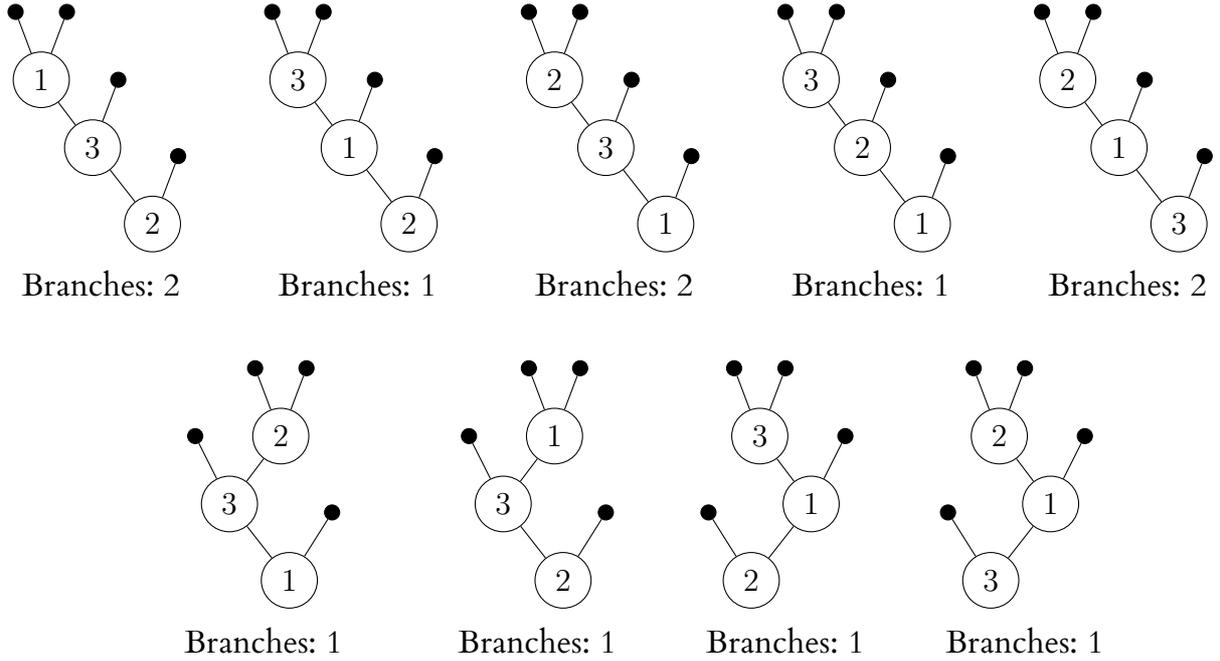
\begin{figure}[!htbp]
    \centering
    \begin{tikzpicture}[scale=0.45]
           \treeb{(0,0)}{$2$}{$3$}{$1$}{2};
           \treeb{(7.5,0)}{$2$}{$1$}{$3$}{1};
           \treeb{(15,0)}{$1$}{$3$}{$2$}{2};
           \treeb{(22.5,0)}{$1$}{$2$}{$3$}{1};
           \treeb{(30,0)}{$3$}{$1$}{$2$}{2};
           
           \treed{(4,-10.5)}{$1$}{$3$}{$2$}{1};
           \treed{(12,-10.5)}{$2$}{$3$}{$1$}{1};
           \treec{(18,-10.5)}{$2$}{$1$}{$3$};
           \treec{(25,-10.5)}{$3$}{$1$}{$2$};
    \end{tikzpicture}
    \caption{Shi trees for $n=3$ that are not Linial.}
    \label{shi}
\end{figure}

\begin{remark}
The Shi trees $\T_{\{0,1\}}(n)$ are in bijection with Cayley trees on $n+1$ vertices. 
Using a decomposition of Cayley trees, one can show that the coefficient of $t^j$ in $\chi_{\mathcal{S}_n}(t)$ is the number of such Cayley trees where the vertex $n+1$ has degree $j$.
\end{remark}

\begin{example}
The Linial arrangement $\mathcal{L}_n$ in $\R^n$ is the deformation $\A_{\{1\}}(n)$.
The trees in $\T_{\{1\}}(n)$, called Linial trees, are those Shi trees that also satisfy the property that any left node whose sibling is a leaf has smaller label than that of its parent.
The Linial trees for $n=3$ are given in \Cref{lin}.
Counting the branches in these trees, we get $\chi_{\mathcal{L}_3}(t)=t^3+3t^2+3t$, which agrees with the known formula for the characteristic polynomial (for example, see \cite[Theorem 4.2]{athanas96}).
\end{example}

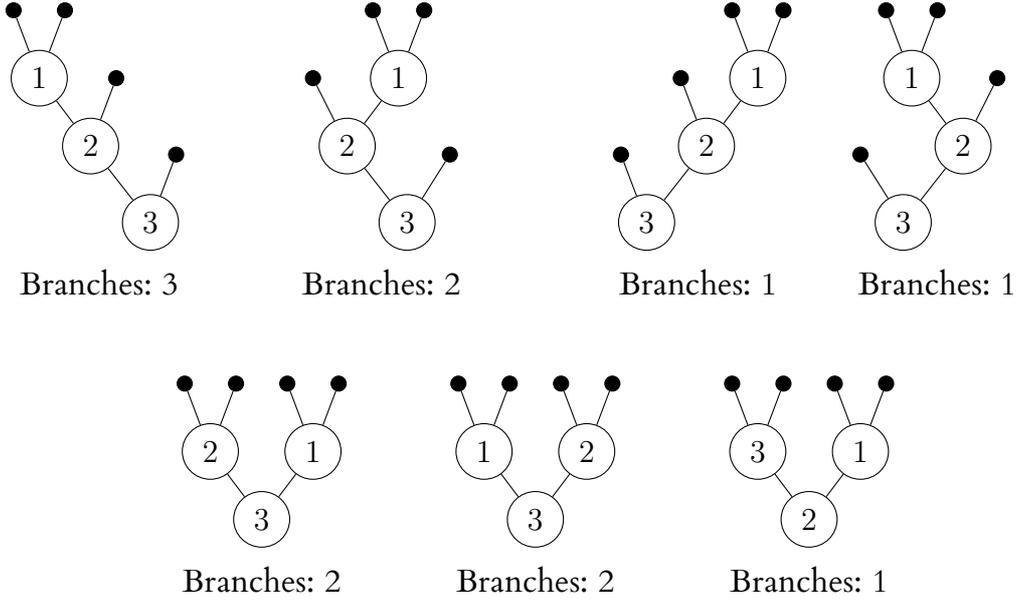
\begin{figure}[!htbp]
    \centering
    \begin{tikzpicture}[scale=0.45]
           \treeb{(0,0)}{$3$}{$2$}{$1$}{3};
           \treed{(7.5,0)}{$3$}{$2$}{$1$}{2};
           \treea{(15,0)}{$3$}{$2$}{$1$};
           \treec{(22.5,0)}{$3$}{$2$}{$1$};
           
           \treee{(2,-11)}{$2$}{$3$}{$1$}{2};
           \treee{(10,-11)}{$1$}{$3$}{$2$}{2};
           \treee{(18,-11)}{$3$}{$2$}{$1$}{1};
    \end{tikzpicture}
    \caption{Linial trees for $n=3$.}
    \label{lin}
\end{figure}

\section{Properties of coefficients}\label{mcat}

In this section, we use the combinatorial interpretation of the coefficients to obtain some of their properties, some of which are not obvious. 
For any transitive set $S$, we use $C(S, n, j)$ to denote the absolute value of the coefficient of $t^j$ in $\chi_{\A_S(n)}(t)$.

\begin{proposition}
For any transitive sets $S' \subseteq S$, we have
\begin{equation*}
    C(S', n, j) \leq C(S, n, j).
\end{equation*}
\end{proposition}

\begin{proof}
As stated in \cite[Remark 3.10]{ber}, we have $\T_S' \subseteq \T_S$. 
When $\operatorname{max}\{|s'| \mid s' \in S'\} < \operatorname{max}\{|s| \mid s \in S\}$, this inclusion can be obtained via adding the required number of leaves to each node of the trees in $\T_S'$. 
This inclusion preserves the number of branches and hence gives us the required result.
\end{proof}

\begin{proposition}
For any transitive set $S$, we have
\begin{equation*}
    C(S, n, j) \leq C(S, n + 1, j + 1).
\end{equation*}
\end{proposition}

\begin{proof}
Given a tree $T \in \T_S(n)$, we construct one in $\T_S(n + 1)$ that has root $n + 1$ with $T$ attached as the leftmost child and all other children being leaves. 
This proves the required result.
\end{proof}

\begin{proposition}
For any transitive set $S$, we have
\begin{equation*}
    C(S, n, j) \leq C(S, n + 1, j).
\end{equation*}
\end{proposition}

\begin{proof}
The result follows by increasing the label of each node and replacing the leftmost leaf with a node labeled $1$ with all leaf children.
\end{proof}

We can derive some more properties for a particular class transitive sets. 
This follows by using a different break-up of trees into connected components from the one presented in \Cref{bstat}.

\begin{proposition}\label{moreconnec}
If $S$ is a transitive set such that $0 \in S$ and there exists some $k \geq 1$ such that $k, -k \in S$, we have
\begin{equation*}
    C(S,n,1) \geq \sum_{j=2}^{n}C(S,n,j).
\end{equation*}
\end{proposition}

\begin{proof}
We have to show that there are more connected trees in $\T_S(n)$ than disconnected trees. 
To prove this, we use \textit{all} the nodes, not just the ones on the trunk, to break-up the tree into \textit{compartments}.

Let the nodes on the trunk of a tree be $v_1,v_2,\ldots,v_k$, where $v_1$ is the root and $v_{i+1}$ is the leftmost child of $v_i$ for any $i \in [k-1]$. 
If the twig corresponding to $v_i$ is the one that contains the label $n$, then the first compartment of the tree consists of the twigs corresponding to the nodes $v_1,\ldots,v_i$. 
If the twig corresponding to $v_j$ is the one that has the node with maximum label among the twigs corresponding to $v_{i+1},\ldots,v_k$, then the second compartment of the tree consists of the twigs corresponding to the nodes $v_{i+1},\ldots,v_j$. 
Continuing this way, we break up the tree into compartments.
For example, the tree in \Cref{binarytree} has $2$ compartments and that in \Cref{built3arytree} has $1$ compartment.

It can be checked that this is a valid break-up of trees into connected components since there is no condition relating a node to its leftmost child since $0 \in S$. 
Hence, $C(S,n,j)$ is the number of trees in $\T_S(n)$ with $j$ compartments. 
In this situation, disconnected trees are those trees where $n$ is not in the last twig of the tree and connected trees are those where it is. 
To prove the result, we have to show that the number of disconnected trees is less than the number of connected trees.

Let $T$ be a disconnected tree. 
First let us suppose that the label $n$ is not in the first twig of $T$. 
Let $i$ be the node on the trunk whose corresponding twig has the node $n$ and $j$ be the parent of $i$. 
Fix some $k \geq 1$ such that $k, -k \in S$ which exists by our hypothesis. 
Let $T_j$ be the subtree of $T$ whose root is the $(k + 1)^{th}$ child of $j$. 
We construct a connected tree $T'$ by breaking off the subtree of $T$ whose root is $i$ and attaching it to $T_j$ by replacing the leftmost leaf of $T_j$ by the node $i$. 
It can be checked that $k, -k \in S$ implies that $T'$ will indeed be in $\T_S(n)$.

\begin{example}
If $T$ is the disconnected tree in \Cref{nnotfirstwigd} and $k = 1$, then $T'$ is the connected tree in \Cref{nnotfirstwigc}.
\end{example}

\begin{figure}[!htbp]
    \centering
    \begin{tikzpicture}[scale=0.8,rotate=180]
    \node (1) [draw=black,circle] at (4,0) {$6$};
    \node (11) [draw=black,circle] at (1-1+3-0.5,0.8) {$7$};
    \draw (1)--(11);
    \node (15) [circle,fill=black,inner sep=2pt] at (0.5-1+3-0.5,0.75-5+4) {};
    \node (16) [circle,fill=black,inner sep=2pt] at (1-1+3-0.5,0.75-5+4) {};
    \draw (15)--(11)--(16);
    \node (2) [draw=black,circle] at (1+1,-2) {$1$};
    \node (b) [circle,fill=black,inner sep=2pt] at (2,-3) {};
    \node (3) [draw=black,circle] at (4,-2) {$3$};
    \node (4) [draw=black,circle] at (7-0.5,-0.5-2+1) {$4$};
    \node (5) [circle,fill=black,inner sep=2pt] at (0.5+1,-3) {};
    \node (6) [circle,fill=black,inner sep=2pt] at (3.5,-3) {};
    \node (7) [circle,fill=black,inner sep=2pt] at (4,-3) {};
    \node (8) [circle,fill=black,inner sep=2pt] at (4.5,-3) {};
    \node (9) [circle,fill=black,inner sep=2pt] at (6.5-0.5,-0.5-3+1) {};
    \node (10) [draw=black,circle] at (10-0.5-0.75,-0.25-4+1) {$5$};
    \node (12) [draw=black,circle] at (7-0.5,-0.5-4+1) {$8$};
    \node (13) [draw=black,circle] at (10-1+0.25-0.5-0.75,-0.5-5+1) {$2$};
    \node (14) [circle,fill=black,inner sep=2pt] at (11-1-0.5-0.75,-0.5-5+1) {};
    \node (18) [circle,fill=black,inner sep=2pt] at (6.5-0.5,-0.5-5+1) {};
    \node (19) [circle,fill=black,inner sep=2pt] at (7-0.5,-0.5-5+1) {};
    \node (20) [circle,fill=black,inner sep=2pt] at (7.5-0.5,-0.5-5+1) {};
    \node (21) [circle,fill=black,inner sep=2pt] at (9-1+0.5+0.25-0.5-0.75,-0.5-6+1) {};
    \node (22) [circle,fill=black,inner sep=2pt] at (10-1+0.25-0.5-0.75,-0.5-6+1) {};
    \node (23) [circle,fill=black,inner sep=2pt] at (11-1-0.5+0.25-0.5-0.75,-0.5-6+1) {};
    \node (24) [circle,fill=black,inner sep=2pt] at (10.5+0.25-0.5-0.75,-0.5-5+1) {};
    \node (25) [circle,fill=black,inner sep=2pt] at (1.5+1,-3) {};
    \draw (5)--(2)--(1)--(4)--(10)--(24);
    \draw (2)--(b);
    \draw (6)--(3)--(7);
    \draw (8)--(3);
    \draw (9)--(4)--(12)--(18);
    \draw (19)--(12)--(20);
    \draw (21)--(13)--(10)--(14);
    \draw (22)--(13)--(23);
    \draw (3)--(1);
    \draw (2)--(25);
    \end{tikzpicture}
    \caption{Disconnected tree with largest label not in first twig.}
    \label{nnotfirstwigd}
\end{figure}
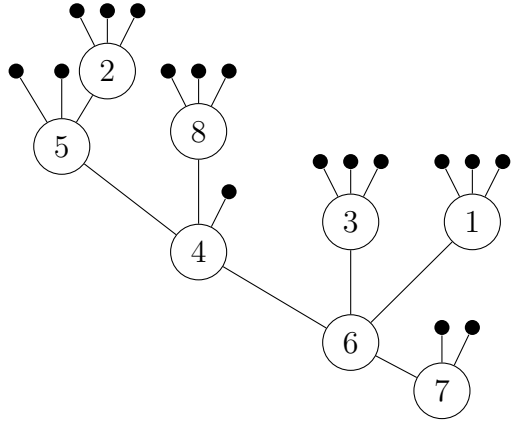
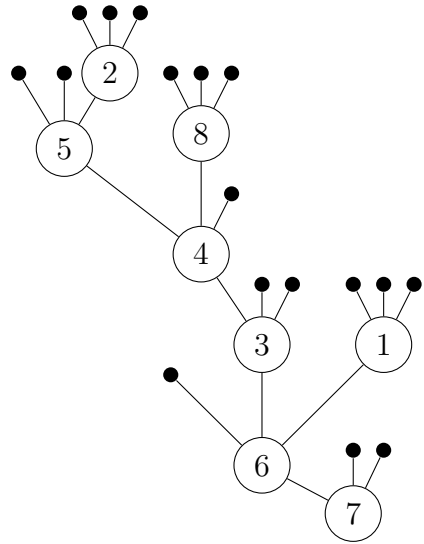
\begin{figure}[!htbp]
    \centering
    \begin{tikzpicture}[scale=0.8,rotate=180]
    \node (1) [draw=black,circle] at (4,0) {$6$};
    \node (11) [draw=black,circle] at (1-1+3-0.5,0.8) {$7$};
    \draw (1)--(11);
    \node (15) [circle,fill=black,inner sep=2pt] at (0.5-1+3-0.5,0.75-5+4) {};
    \node (16) [circle,fill=black,inner sep=2pt] at (1-1+3-0.5,0.75-5+4) {};
    \draw (15)--(11)--(16);
    \node (2) [draw=black,circle] at (1+1,-2) {$1$};
    \node (b) [circle,fill=black,inner sep=2pt] at (2,-3) {};
    \node (3) [draw=black,circle] at (4,-2) {$3$};
    \node (5) [circle,fill=black,inner sep=2pt] at (0.5+1,-3) {};
    \node (6) [circle,fill=black,inner sep=2pt] at (3.5,-3) {};
    \node (7) [circle,fill=black,inner sep=2pt] at (4,-3) {};
    \node (8) [circle,fill=black,inner sep=2pt] at (5.5,-1.5) {};
    \node (25) [circle,fill=black,inner sep=2pt] at (1.5+1,-3) {};
    
    \node (4) [draw=black,circle] at (7-0.5-2+0.5,-0.5-1.5-0.5-2+1) {$4$};
    \node (9) [circle,fill=black,inner sep=2pt] at (6.5-0.5-2+0.5,-0.5-1.5-0.5-3+1) {};
    \node (10) [draw=black,circle] at (10-0.5-0.75-2+0.5,-0.5-1.5-0.25-4+1) {$5$};
    \node (12) [draw=black,circle] at (7-0.5-2+0.5,-0.5-1.5-0.5-4+1) {$8$};
    \node (13) [draw=black,circle] at (10-1+0.25-0.5-0.75-2+0.5,-0.5-1.5-0.5-5+1) {$2$};
    \node (14) [circle,fill=black,inner sep=2pt] at (11-1-0.5-0.75-2+0.5,-0.5-1.5-0.5-5+1) {};
    \node (18) [circle,fill=black,inner sep=2pt] at (6.5-0.5-2+0.5,-0.5-1.5-0.5-5+1) {};
    \node (19) [circle,fill=black,inner sep=2pt] at (7-0.5-2+0.5,-0.5-1.5-0.5-5+1) {};
    \node (20) [circle,fill=black,inner sep=2pt] at (7.5-0.5-2+0.5,-0.5-1.5-0.5-5+1) {};
    \node (21) [circle,fill=black,inner sep=2pt] at (9-1+0.5+0.25-0.5-0.75-2+0.5,-0.5-1.5-0.5-6+1) {};
    \node (22) [circle,fill=black,inner sep=2pt] at (10-1+0.25-0.5-0.75-2+0.5,-0.5-1.5-0.5-6+1) {};
    \node (23) [circle,fill=black,inner sep=2pt] at (11-1-0.5+0.25-0.5-0.75-2+0.5,-0.5-1.5-0.5-6+1) {};
    \node (24) [circle,fill=black,inner sep=2pt] at (10.5+0.25-0.5-0.75-2+0.5,-0.5-1.5-0.5-5+1) {};
    
    \draw (5)--(2)--(1);
    \draw (4)--(10)--(24);
    \draw (2)--(b);
    \draw (6)--(3)--(7);
    \draw (4)--(3);
    \draw (1)--(8);
    \draw (9)--(4)--(12)--(18);
    \draw (19)--(12)--(20);
    \draw (21)--(13)--(10)--(14);
    \draw (22)--(13)--(23);
    \draw (3)--(1);
    \draw (2)--(25);
    \end{tikzpicture}
    \caption{Connected tree associated to the tree in \Cref{nnotfirstwigd}.}
    \label{nnotfirstwigc}
\end{figure}

To obtain $T$ back from $T'$, we first need the following definition from \cite{ber}.
We define $\operatorname{drift}(u)=0$ when $u$ is the root of a tree and for any other vertex $v$ with parent $w$, define $\operatorname{drift}(v)=\operatorname{lsib}(v)+\operatorname{drift}(w)$.

To obtain $T$ back from $T'$, we note that $i$ is the vertex with drift $k$ that is furthest from the root in the unique path from the root to the node $n$. 
Now $T$ is obtained by breaking off the subtree of $T'$ with root $i$ and replacing the leftmost leaf of $T'$ with this subtree.

Next, suppose that the root of $T$ is $i$ and the label $n$ is in the first twig.
Let $j$ be the last node on the trunk, which is necessarily different from $i$ since $T$ is disconnected.
Just as before, let $T_j$ be the subtree of $T$ whose root is the $(k + 1)^{th}$ of $j$ and construct $T'$ by removing the first twig of $T$ and using it to replace the leftmost leaf of $T_j$.

\begin{example}
If $T$ is the disconnected tree on the left in \Cref{nfirstwig} and $k=1$, then $T'$ is the connected tree on the right.
\end{example}

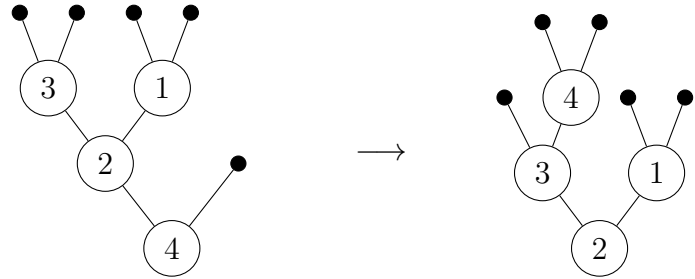
\begin{figure}[!htbp]
\centering
    \begin{tikzpicture}[xscale=-1,scale=0.5]
        \node [circle,draw=black](1) at (-0.25,-0.25) {$4$};
        \node [circle,draw=black,inner sep=2pt,fill=black](2) at (-2,2) {};
        \node [circle,draw=black](3) at (1.5,2) {$2$};
        \node [circle,draw=black](4) at (3,4) {$3$};
        \node [circle,draw=black,inner sep=2pt,fill=black] (4') at (3-0.75,6) {};
        \node [circle,draw=black,inner sep=2pt,fill=black] (4'') at (3+0.75,6) {};
        \node [circle,draw=black](5) at (0,4) {$1$};
        \node [circle,draw=black,inner sep=2pt,fill=black](6) at (0.75,6) {};
        \node [circle,draw=black,inner sep=2pt,fill=black](7) at (-0.75,6) {};
        \draw (1)--(2);
        \draw (1)--(3);
        \draw (3)--(4);
        \draw (3)--(5);
        \draw (5)--(6);
        \draw (5)--(7);
        \draw (4')--(4)--(4'');
        
        \node at (-5.75,2.25) {$\longrightarrow$};
        
        \node [circle,draw=black](3a) at (1.5-13,-2.25+2) {$2$};
        \node [circle,draw=black](4a) at (3-13,-2.25+4) {$3$};
        \node [circle,draw=black] (4'a) at (3-0.75-13,-2.25+6) {$4$};
        \node [circle,draw=black,inner sep=2pt,fill=black] (n) at (3-0.75-0.75-13,-2.25+8) {};
        \node [circle,draw=black,inner sep=2pt,fill=black] (m) at (3-0.75+0.75-13,-2.25+8) {};
        \node [circle,draw=black,inner sep=2pt,fill=black] (4''a) at (3+1-13,-2.25+6) {};
        \node [circle,draw=black](5a) at (0-13,-2.25+4) {$1$};
        \node [circle,draw=black,inner sep=2pt,fill=black](6a) at (0.75-13,-2.25+6) {};
        \node [circle,draw=black,inner sep=2pt,fill=black](7a) at (-0.75-13,-2.25+6) {};
        \draw (3a)--(4a);
        \draw (3a)--(5a);
        \draw (5a)--(6a);
        \draw (5a)--(7a);
        \draw (n)--(4'a)--(m);
        \draw (4'a)--(4a)--(4''a);
    \end{tikzpicture}
    \caption{Disconnected tree with largest label in the first twig and associated connected tree.}
    \label{nfirstwig}
\end{figure}

We can determine the root $i$ of $T$ from $T'$ just as in the previous case, \textit{i.e.}, $i$ is the vertex with drift $k$ that is furthest from the root in the unique path from the root to the node $n$.
Now $T$ is obtained from $T'$ by breaking off the subtree with root $i$ and making it the first twig.

Note that this is an injection from the set of disconnected trees to the set of connected trees.
This is because we can obtain $T$ from $T'$ in each case above and there is no overlap of the connected trees in the two cases.
This is because, in the first case, the leftmost child of the node $i$ is always a node, whereas it is always a leaf in the second case.
This proves the required result.
\end{proof}

\begin{proposition}\label{coeffinc}
If $S$ is a transitive set such that $0 \in S$ and there exists some $k \geq 1$ such that $k, -k \in S$, we have
\begin{equation*}
    C(S,n,j) \geq C(S,n,j+1).
\end{equation*}
\end{proposition}
\begin{proof}
Using the ``connected structure'' property of branches and considering the size of the branch containing the label $n$, we get for $j \geq 2$,
\begin{equation*}
    C(S,n,j) = \sum_{k=1}^{n-1}\binom{n-1}{k-1}C(S,k,1)C(S,n-k,j-1).
\end{equation*}
By induction on $j$, this shows that it is enough to prove that $C(S,n,1) \geq C(S,n,2)$ for all $n \geq 1$.
This follows from \Cref{moreconnec}.
\end{proof}

\subsection{Extended Catalan arrangement}

We now focus on the case when $S=\{-m,-m+1,\ldots,m-1,m\}$ for some $m \geq 1$.
The corresponding arrangement $\A_S(n)$ is called the $m$-Catalan arrangement in $\R^n$. 
We let $C(m, n, j)$ denote the absolute value of the coefficient of $t^j$ in $\chi_{\A_{S}(n)}(t)$. 
Here $\T_S(n)=\T^{(m)}(n)$ and hence, from \Cref{statthm}, $C(m, n, j)$ is the number of $(m + 1)$-ary trees with $n$ nodes and $j$ branches. 
We now compute expressions for $C(m, n, j)$ using this combinatorial interpretation.

\begin{proposition}
We have
\begin{equation*}
    C(m,n,j) = \sum_{k=j}^nT_m(n,k)\binom{n}{k}c(k,j)(n-k)!
\end{equation*}
where
\begin{itemize}
    \item $c(k,j)$ is the number of permutations of $[k]$ with $j$ right-to-left maxima (unsigned Stirling number of first kind), and
    \item $T_m(n,k)$ is the number of unlabeled $(m+1)$-ary trees with $n$ nodes, $k$ of which are on its trunk, given by
    \begin{equation*}
        \frac{mk}{(m+1)n-k}\binom{(m+1)n-k}{n-k}.
    \end{equation*}
\end{itemize}
\end{proposition}

\begin{proof}
Specifying a tree in $\T^{(m)}(n)$ with $j$ branches is done by
\begin{enumerate}
    \item choosing an unlabeled $(m+1)$-ary tree with $n$ nodes and $k \geq j$ nodes on its trunk,
    \item choosing the $k$ labels from $[n]$ for the trunk,
    \item labeling the trunk so that it has $j$ right-to-left maxima, and
    \item using the rest of the labels for the non-trunk nodes.
\end{enumerate}
The formula for $T_m(n,k)$ follows from the bijection between trees and Dyck paths given in \cite[Section 8]{ber} and the discussion about returns in Dyck paths in \cite[Page 4]{rethill}. 
One can check that trunk nodes correspond to returns in the Dyck path under this bijection.
\end{proof}

\begin{proposition}
We also have for any $m,n,j \geq 1$,
\begin{equation*}
    C(m,n,j)=\sum_{k=j}^n (-1)^{k-j}B_m(n,k)c(k,j)
\end{equation*}
where
\begin{equation*}
    B_m(n,k) = \frac{(n-1)!}{(k-1)!}\binom{(m+1)n}{n-k}.
\end{equation*}
\end{proposition}

\begin{proof}
It can be checked, for example using \cite[Theorem 5.3.10]{ec2}, that $B_m(n,k)$ is the number of ways to partition $[n]$ into $k$ blocks and associate to each block $B$ a tree in $\T^{(m)}(B)$. 
Using the fact that branches give the trees an exponential structure, another way to partition $[n]$ into $k$ blocks and associate to each block $B$ a tree in $\T^{(m)}(B)$ is to
\begin{enumerate}
    \item select a collection of $i \geq k$ connected trees whose label sets partition $[n]$ (which is the same as choosing a tree in $\T_i^{(m)}(n)$),
    \item partition these trees into $k$ sets, and
    \item combine the connected trees in each of these $k$ sets to form $k$ trees. 
\end{enumerate}
This shows
\begin{equation*}
    B_m(n,k)=\sum_{i=k}^{n}C(m,n,i)S(i,k)
\end{equation*}
where $S(i,k)$ is the number of ways to partition the set $[i]$ into $k$ blocks, \textit{i.e.}, they are Stirling numbers of the second kind.
The result now follows by M\"obius inversion, since by \cite[Proposition 1.9.1]{ec1}, we have
\begin{equation*}
    \sum_{k \geq j}(-1)^{k-j}S(i,k)c(k,j) = \begin{cases}
                0, & \text{if }i \neq j\\
                1, & \text{if }i = j.
            \end{cases}
\end{equation*}
\end{proof}

\begin{remark}
The triangle of numbers $C(1,n,j)$ is listed in the OEIS \cite{oeis} as \href{https://oeis.org/A038455}{A038455}.
For $m \geq 2$, the triangle $C(m,n,j)$ does not seem to be listed.
\end{remark}

We now derive some properties of $C(m,n,j)$. 
First, we list properties that are consequences of the general properties we have already seen.

\begin{proposition}
For any $m,n,j \geq 1$, we have:
\begin{enumerate}
\setlength{\itemsep}{0.35cm}
    \item $C(m,n,j) \leq C(m+1,n,j)$
    \item $C(m,n,j) \leq C(m,n+1,j)$
    \item $C(m,n,1) \geq \sum\limits_{k=2}^{n}C(m,n,k)$
    \item $C(m,n,j) \geq C(m,n,j+1)$
\end{enumerate}
\end{proposition}

We now list some properties that are specific to the case of the extended Catalan arrangement.

\begin{proposition}
For any $m,n \geq 1$, we have
\begin{equation*}
    C^{(m)}(n):=\sum_{j=1}^{n}C(m,n,j) = \frac{n!}{mn+1}\binom{(m+1)n}{n}.
\end{equation*}
\end{proposition}
\begin{proof}
This follows from the known formula for $|\T^{(m)}(n)|$ (see \cite[Section 5.3]{ec2}).
\end{proof}

\begin{proposition}
For any $m,n\geq 1$, we have
\begin{equation*}
    C^{(m)}(n) \leq C(m+1,n,1).
\end{equation*}
\end{proposition}
\begin{proof}
Adding a leaf as the first child of each node in a labeled $(m+1)$-ary tree gives the required result.
\end{proof}

\begin{proposition}
For any $m,n \geq 1$, we have
\begin{equation*}
    C^{(m)}(n) \leq C(m,n+1,1).
\end{equation*}
\end{proposition}
\begin{proof}
The result follows by replacing the leftmost leaf of a tree in $\T^{(m)}(n)$ with a node labeled $n+1$ with all children being leaves.
\end{proof}

There are several combinatorial objects that correspond to the regions of the extended Catalan arrangement (especially in the case $m=1$, see \cite{cat}).
One such is the generalized Dyck paths.
We now describe a corresponding statistic for these Dyck paths.

A labeled $m$-Dyck path on $[n]$ is a sequence of $(m+1)n$ terms where
\begin{enumerate}
    \item $n$ terms are `$+m$',
    \item $mn$ terms are `$-1$',
    \item the sum of any prefix of the sequence is non-negative, and
    \item each $+m$ term is given a distinct label from $[n]$.
\end{enumerate}

A labeled $m$-Dyck path on $[n]$ can be drawn in $\mathbb{R}^2$ in the natural way.
Start the path at $(0,0)$, read the labeled $m$-Dyck path and for each term move by $(1,m)$ if it is $+m$ and by $(1,-1)$ if it is $-1$.
Also, label each $+m$ step with its corresponding label in $[n]$.

A Dyck path breaks up into \textit{primitive} parts based on when it touches the $x$-axis.
If a labeled Dyck path has $k$ primitive parts, then we break the path into \textit{compartments} as follows.
If the number $n$ is in the $i_1^{th}$ primitive part, then the primitive parts up to the $i_1^{th}$ form the first compartment.
Let $j$ be the largest number in $[n] \setminus A$ where $A$ is the set of numbers in first compartment.
If $j$ is in the $i_2^{th}$ primitive part then the primitive parts after the $i_1^{th}$ up to the $i_2^{th}$ form the second compartment.
Continuing this way, we break up a labeled Dyck path into compartments.

\begin{example}
The labeled $1$-Dyck path on $[7]$ given in \Cref{dyck} has $3$ primitive parts and $2$ compartments.
\end{example}

\begin{figure}[!htbp]
    \centering
    \begin{tikzpicture}[xscale=0.65]
    \draw [thin,green] (8,0)--(22,0);
    \node (8) [circle,inner sep=2pt,fill=black] at (8,0) {};
    \node (9) [circle,inner sep=2pt,fill=black] at (9,1) {};
    \node (10) [circle,inner sep=2pt,fill=black] at (10,2) {};
    \node (11) [circle,inner sep=2pt,fill=black] at (11,3) {};
    \node (12) [circle,inner sep=2pt,fill=black] at (12,2) {};
    \node (13) [circle,inner sep=2pt,fill=black] at (13,1) {};
    \node (14) [circle,inner sep=2pt,fill=black] at (14,2) {};
    \node (15) [circle,inner sep=2pt,fill=black] at (15,1) {};
    \node (16) [circle,inner sep=2pt,fill=black] at (16,0) {};
    \node (17) [circle,inner sep=2pt,fill=black] at (17,1) {};
    \node (18) [circle,inner sep=2pt,fill=black] at (18,2) {};
    \node (19) [circle,inner sep=2pt,fill=black] at (19,1) {};
    \node (20) [circle,inner sep=2pt,fill=black] at (20,0) {};
    \node (21) [circle,inner sep=2pt,fill=black] at (21,1) {};
    \node (22) [circle,inner sep=2pt,fill=black] at (22,0) {};
    \draw (8)--node[left] {$4$}(9)--node[left] {$7$}(10)--node[left] {$2$}(11)--(12)--(13)--node[left] {$6$}(14)--(15)--(16)--node[left] {$3$}(17)--node[left] {$1$}(18)--(19)--(20)--node[left] {$5$}(21)--(22);
    \draw [blue,edgee] (8) to [bend right] (16);
    \draw [blue,edgee] (16) to [bend right] (22);
    \end{tikzpicture}
    \caption{A labeled $1$-Dyck path with compartments specified.}
    \label{dyck}
\end{figure}
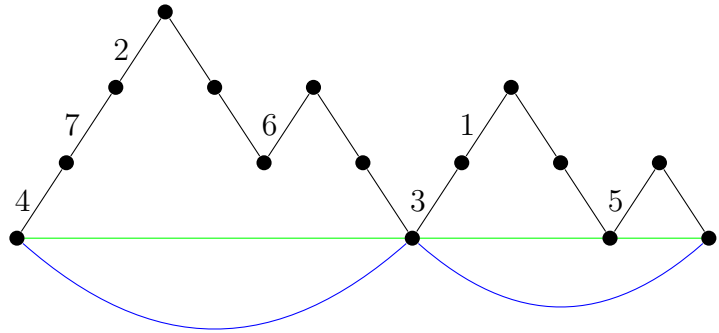

It can be checked that this is a valid break-up of Dyck paths into connected structures.
Thus have the following.

\begin{theorem}
The number of labeled $m$-Dyck paths on $[n]$ with $j$ compartments is $C(m,n,j)$.
\end{theorem}

We say that a labeled Dyck path has $j$ right-to-left maxima if the string of labels before its first down step has $j$ right-to-left maxima.
For example, the string of labels before the first down step in the Dyck path in \Cref{dyck} is $4,7,2$.
Hence, it has $2$ right-to-left maxima.

Using the bijection between labeled trees and labeled Dyck paths given in \cite{ber} and the result in \Cref{bstat}, we get another statistic on labeled Dyck paths with the same distribution.

\begin{theorem}
The number of labeled $m$-Dyck paths on $[n]$ with $j$ right-to-left maxima is $C(m,n,j)$.
\end{theorem}

\section{Concluding remarks}

We note that a combinatorial interpretation for the coefficients of the characteristic polynomial of the Linial arrangement is already given in \cite[Corollary 4.2]{stanint}.
This is in terms of alternating trees.

For various deformations of the braid arrangement, expressions for the characteristic polynomials are known (for example, see \cite{ard,athanas96}).
Hence, for transitive sets $S$, these can be used to extract coefficients and hence give formulas for the number of trees in $\T_S$ according to number nodes and branches.

\section{Acknowledgements}

The authors are partially supported by a grant from the Infosys Foundation.
We thank the anonymous referee for helpful suggestions.

\end{document}